\documentclass[11pt,a4paper]{article}%
\usepackage{amssymb,amsmath,amsfonts,amsthm,array,bm,color}%
\usepackage{amsmath}%
\usepackage{amsfonts}%
\usepackage{amssymb}%
\usepackage{graphicx}
\providecommand{\U}[1]{\protect\rule{.1in}{.1in}}
\newtheorem{theorem}{Theorem}[section]
\newtheorem{lemma}[theorem]{Lemma}

\newtheorem{proposition}[theorem]{Proposition}
\newtheorem{remark}[theorem]{Remark}

\def\<{\langle}
\def\>{\rangle}
\def\d{{\rm d}}
\def\L{\mathcal{L}}
\def\div{{\rm div}}
\def\T{\mathbb{T}}
\def\R{\mathbb{R}}
\begin{document}

\title{Euler-Lagrangian approach to 3D stochastic Euler equations}
\author{Franco Flandoli\footnote{Email: franco.flandoli@sns.it. Scuola Normale Superiore of Pisa, Italy.} \ and
Dejun Luo\footnote{Email: luodj@amss.ac.cn. Key Laboratory of Random Complex Structures and Data Sciences, Academy of Mathematics and Systems Science, Chinese Academy of Sciences, Beijing 100190, China, and School of Mathematical Sciences, University of the Chinese Academy of Sciences, Beijing 100049, China. }}

\maketitle

\begin{abstract}
3D stochastic Euler equations with a special form of multiplicative noise are
considered. A Constantin-Iyer type representation in Euler-Lagrangian form
is given, based on stochastic characteristics. Local existence and uniqueness
of solutions in suitable H\"{o}lder spaces is proved from the Euler-Lagrangian formulation.
\end{abstract}



%\maketitle
\makeatletter
%'@' is now a normal "letter" for TeX
\renewcommand\theequation{\thesection.\arabic{equation}}
\@addtoreset{equation}{section} \makeatother
%'@' is restored as a "non-letter" character for TeX

\section{Introduction}

Recently the vorticity equation (Euler type) with noise%
\begin{equation}
\mathrm{d}\omega_{t}+\left(  u_{t}\cdot\nabla\omega_{t}-\omega_{t}\cdot\nabla
u_{t}\right)  \mathrm{d}t+\sum_{k}\left(  \sigma_{k}\cdot\nabla\omega
_{t}-\omega_{t}\cdot\nabla\sigma_{k}\right)  \circ\mathrm{d}W_{t}%
^{k}=0\label{eq vort}%
\end{equation}
with $\operatorname{div}u_{t}=0$, $\operatorname{curl}u_{t}=\omega_{t}$,
$\operatorname{div}\sigma_{k}=0$, has been studied by Darryl Holm and other
authors, see \cite{Holm, GBH, CFH}. It can be written as%
\begin{equation}
\mathrm{d}\omega_{t}+\mathcal{L}_{u_{t}}\omega_{t}\,\mathrm{d}t+\sum
_{k}\mathcal{L}_{\sigma_{k}}\omega_{t}\circ\mathrm{d}W_{t}^{k}%
=0,\label{eq vort-1}%
\end{equation}
where $\mathcal{L}_{u_{t}}\omega_{t}=[u_{t},\omega_{t}]=u_{t}\cdot\nabla
\omega_{t}-\omega_{t}\cdot\nabla u_{t}$ is the Lie derivative. Equations
related to fluid dynamics with multiplicative noise appeared in several other
works, see for instance \cite{BrzCapFla, FlaGat, FalkGaVer, MikuRoz, FlaSF} and many others. However, the geometric structure
in (\ref{eq vort}) has special properties, revealed also by the present work.

A first intermediate question we address is finding the noise form when the
equation is rewritten in the velocity-pressure variables $\left(  u,p\right)
$, instead of the velocity-vorticity variables $\left(  u,\omega\right)  $.
The dual operators $\mathcal{L}_{\sigma_{k}}^{\ast}$ appear. This is an
intermediate step in order to investigate the main topic of this work, namely
the Euler-Lagrangian formulation, called also Constantin-Iyer representation
after \cite{Cons, Cons-Iyer}; among related works, see for instance
\cite{Zhang, FangLuo}. We prove both the $\left(  u,p\right)  $-formulation and the
Euler-Lagrangian one in Proposition \ref{2-prop}. The Euler-Lagrangian form is
then used to prove a local in time existence and uniqueness result for
solutions in suitable H\"{o}lder spaces, new for equation (\ref{eq vort}). At
the end of the paper we heuristically digress on potential singularities from
the viewpoint of this stochastic model and its Euler-Lagrangian formulation,
adding some remarks to the discussion of P. Constantin  \cite[Section 5]{Cons}.
Finally we remark that we restrict ourselves to the three dimensional torus $\mathbb{T}^{3}$ for simplicity,
though the results remain valid in more general settings under suitable conditions.

\section{Equation for the velocity and its representation}

Let $\mathcal{L}_{\sigma_{k}}^{\ast}$ be the adjoint operator of
$\mathcal{L}_{\sigma_{k}}$ with respect to the inner product in $L^{2}%
(\mathbb{T}^{3},\mathbb{R}^{3})$:
\[
\<\mathcal{L}_{\sigma_{k}}^{\ast}v,w\>_{L^{2}}=-\<v,\mathcal{L}_{\sigma_{k}%
}w\>_{L^{2}}.
\]
Since $\sigma_{k}$ is divergence free, we have
\[
\mathcal{L}_{\sigma_{k}}^{\ast}v=\sigma_{k}\cdot\nabla v+(\nabla\sigma
_{k})^{\ast}v.
\]
Note that if the vector fields $u$ and $v$ are divergence free, then
$\mathrm{div}(\mathcal{L}_{u}v)=0$, but this is not necessarily true for
$\mathcal{L}_{u}^{\ast}v$.

Consider the following equation of characteristics, associated to
(\ref{eq vort}):%
\begin{equation}
\mathrm{d}X_{t}=u_{t}\left(  X_{t}\right)  \mathrm{d}t+\sum_{k}\sigma
_{k}\left(  X_{t}\right)  \circ\mathrm{d}W_{t}^{k}.\label{SDE}%
\end{equation}
Denote by $X_{t}$ also the stochastic flow associated to this equation, when
defined.

\begin{proposition}
\label{2-prop} For sufficiently smooth solutions, the equation, corresponding
to (\ref{eq vort}), for the velocity $u_{t}$ is%
\begin{equation}
\mathrm{d}u_{t}+\left(  u_{t}\cdot\nabla u_{t}+\nabla p_{t}\right)
\mathrm{d}t+\sum_{k}\mathcal{L}_{\sigma_{k}}^{\ast}u_{t}\circ\mathrm{d}%
W_{t}^{k}=0,\label{eq vel}%
\end{equation}
with the representation
\begin{equation}%
\begin{cases}
\aligned {\rm d}X_{t} & =\sum_{k}\sigma_{k}(X_{t})\circ\mathrm{d}W_{t}%
^{k}+u_{t}(X_{t})\,\mathrm{d}t,\\
u_{t}(x) & =\mathbf{P}\big[(\nabla X_{t}^{-1})^{\ast}\,u_{0}(X_{t}%
^{-1})\big](x),\endaligned
\end{cases}
\label{representation}%
\end{equation}
where $\mathbf{P}$ is the Leray--Hodge projection and $\ast$ means the transposition of matrices.
\end{proposition}

\begin{remark} {\rm
The stochastic equation \eqref{eq vel} is understood as follows: for any smooth divergence free field $v$ on $\mathbb T^3$,
  $$\aligned
  \<u_t,v\>_{L^2} &= \<u_0,v\>_{L^2} - \int_0^t \<u_s\cdot\nabla u_s, v\>_{L^2} \,\d s - \sum_k\int_0^t \<\L_{\sigma_k}^\ast u_s, v\>_{L^2} \circ\d W^k_s\\
  &= \<u_0,v\>_{L^2} + \int_0^t \<u_s, u_s\cdot\nabla v\>_{L^2} \,\d s + \sum_k\int_0^t \<u_s, \L_{\sigma_k} v\>_{L^2} \circ\d W^k_s.
  \endaligned$$
Since $\L_{\sigma_k} v$ is divergence free, we have
  $$\d \<u_s, \L_{\sigma_k} v\>_{L^2}\cdot \d W^k_s =\<u_s, \L_{\sigma_k}^2 v\>_{L^2}\, \d s.$$
Hence, we get
  \begin{equation}\label{stoch-euler-3.1}
  \aligned
  \<u_t, v\>_{L^2} &= \<u_0,v\>_{L^2} +\sum_k\int_0^t \<u_s, \L_{\sigma_k} v\>_{L^2}\,\d W^k_s\\
  &\hskip13pt +\frac12\sum_k\int_0^t \<u_s, \L_{\sigma_k}^2 v\>_{L^2}\, \d s + \int_0^t \<u_s, u_s\cdot\nabla v\>_{L^2}\, \d s.
  \endaligned
  \end{equation}  }
\end{remark}

\begin{proof}[Proof of Proposition \ref{2-prop}]
Let us show the first fact. Using $\operatorname{curl}u_{t}=\omega_{t}$, we have the vector identities%
\[
\omega_{t}\times u_{t}=-\frac{1}{2}\nabla\left(  u_{t}\cdot u_{t}\right)
+u_{t}\cdot\nabla u_{t}%
\]%
and
\begin{align*}
\omega_{t}\times\sigma_{k}  & =-\nabla\left(  \sigma_{k}\cdot u_{t}\right)
+\left(  D\sigma_{k}\right)  u_{t}+\sigma_{k}\cdot\nabla u_{t}\\
& =-\nabla\left(  \sigma_{k}\cdot u_{t}\right)  +\mathcal{L}_{\sigma_{k}%
}^{\ast}u_{t}.
\end{align*}
Hence, equation (\ref{eq vel}) can be rewritten as
\[
\d u_{t}+\left(  \omega_{t}\times u_{t}+\nabla\widetilde{p}_{t}\right)
\d t+\sum_{k}\left(  \omega_{t}\times\sigma_{k} + \nabla (\sigma_{k}\cdot u_{t})\right)  \circ \d W_{t}^{k}=0, %
\]
where $\widetilde{p}_{t}$ is a new pressure. Taking $\operatorname{curl}$ and
using the facts%
\begin{align*}
\operatorname{curl}\left(  \omega_{t}\times u_{t}\right) =\mathcal{L}%
_{u_{t}}\omega_{t},\quad
\operatorname{curl} (\nabla\widetilde{p}_{t}) =0,\quad
\operatorname{curl}\left(  \omega_{t}\times\sigma_{k}\right)
=\mathcal{L}_{\sigma_{k}}\omega_{t},%
\end{align*}
we get the equation \eqref{eq vort-1}.

Next we prove the second assertion. Let $u_t$ be expressed as in \eqref{representation}. For any divergence free vector field $v$, we have
  $$\aligned
  \<u_t, v\>_{L^2}&= \int \big\<(\nabla X_t^{-1})^\ast\, u_0(X_t^{-1}), v\big\>\,\d x \\
  &= \int \big\<u_0(X_t^{-1}), (\nabla X_t^{-1})\, v\big\>\,\d x\\
  &= \int \big\<u_0, (\nabla X_t^{-1}(X_t))\, v(X_t)\big\>\,\d x,
  \endaligned$$
where in the last step we used the change of variable formula and the fact that $X_t$ preserves the volume measure. Recall that $(\nabla X_t^{-1}(X_t))\, v(X_t) =(X_t^{-1})_\ast v$ is the pull-back of $v$ by the flow $X_t$, thus we obtain
  \begin{equation}\label{key-identity}
  \<u_t, v\>_{L^2} =\big\<u_0, (X_t^{-1})_\ast v\big\>_{L^2}\quad \mbox{for all } t\geq 0.
  \end{equation}
This equality holds for any divergence free vector field $v$.

Since $X_t$ is the flow generated by the SDE in \eqref{SDE}, by Kunita's formula (see \cite[p. 265]{Kunita}),
  $$\aligned (X_t^{-1})_\ast v &= v +\sum_k\int_0^t (X_s^{-1})_\ast (\L_{\sigma_k} v)\, \d W^k_s\\
  &\hskip13pt +\int_0^t \Big[(X_s^{-1})_\ast (\L_{u_s} v) +\frac12\sum_k (X_s^{-1})_\ast (\L_{\sigma_k}^2 v)\Big] \d s.
  \endaligned$$
Substituting this expression into \eqref{key-identity} yields
  $$\aligned
  \<u_t, v\>_{L^2} &= \<u_0, v\>_{L^2} + \sum_k\int_0^t \big\<u_0, (X_s^{-1})_\ast (\L_{\sigma_k} v)\big\>_{L^2}\, \d W^k_s\\
  &\hskip13pt + \int_0^t \Big[\big\<u_0, (X_s^{-1})_\ast (\L_{u_s} v) \big\>_{L^2} + \frac12\sum_k \big\<u_0, (X_s^{-1})_\ast (\L_{\sigma_k}^2 v)\big\>_{L^2}\Big]\, \d s.
  \endaligned $$
Since the vector fields $\L_{\sigma_k} v,\, \L_{u_s} v$ and $\L_{\sigma_k}^2 v$ are all divergence free, we apply \eqref{key-identity} and get
  \begin{equation*}
  \<u_t, v\>_{L^2} = \<u_0,v\>_{L^2} +\sum_k\int_0^t \<u_s, \L_{\sigma_k} v\>_{L^2}\,\d W^k_s +\int_0^t \Big[\<u_s, \L_{u_s} v\>_{L^2} + \frac12 \sum_k \<u_s,  \L_{\sigma_k}^2 v\>_{L^2} \Big]\d s.
  \end{equation*}
Note that
  $$\<u_s, \L_{u_s} v\>_{L^2}= \<u_s, u_s\cdot\nabla v\>_{L^2} -\<u_s, v\cdot\nabla u_s\>_{L^2}= \<u_s, u_s\cdot\nabla v\>_{L^2},$$
therefore, we obtain \eqref{stoch-euler-3.1}.
\end{proof}

\section{Local existence of the representation \eqref{representation}}

In this section we aim at proving the local existence of the system \eqref{representation}, by following the arguments in \cite{Cons, Iyer}. First we introduce some notations about H\"older (semi-)norms. For a function or vector field $u$ defined on $\mathbb T^3$ and $\alpha\in (0,1),\, l\in \mathbb N$,
  $$\aligned |u|_\alpha &=\sup_{x,y\in \mathbb T^3} \frac{|u(x)- u(y)|}{|x-y|^\alpha}, \\
  \|u\|_l&= \sum_{|m|\leq l} \sup_{x\in \mathbb T^3} |\partial^m u(x)|,\\
  \|u\|_{l,\alpha} &= \|u\|_l +\sum_{|m|=l} |\partial^m u|_\alpha,
  \endaligned$$
where $\partial^m$ denotes the derivative with respect to the multi-index $m\in \mathbb N^3$. Note that $\|\cdot\|_0$ is the usual supremum norm. We denote by $C^l$ and $C^{l, \alpha}$ the H\"older spaces with norms $\|u\|_l$ and $\|u\|_{l,\alpha}$, respectively.

We use the idea of \cite{BM, OP} to solve the SDE in the system \eqref{representation}. More precisely, we first solve the equation without drift:
  \begin{equation}\label{SDE-1}
  \d\varphi_t = \sum_k \sigma_k(\varphi_t)\circ \d W^k_t,\quad \varphi_0={\rm I},
  \end{equation}
where ${\rm I}$ is the identity diffeomorphism of $\mathbb T^3$. Under the assumption that $\sum_k \|\sigma_k\|_{l+3, \alpha'}^2 <\infty$ for some $l\in \mathbb N$ and $\alpha'\in (0,1)$, the above equation generates a stochastic flow $\{\varphi_t\}_{t\geq 0}$ of $C^{l+2,\alpha}$-diffeomorphisms on $\mathbb T^3$, where $\alpha\in (0,\alpha')$.

In this section we denote by $\omega$ a generic random element in a probability space $\Omega$; there will be no confusion with the notation of vorticity, since the latter does not appear in the current section. For a given random vector field $u:\Omega\times [0,T]\times \mathbb T^3 \to \R^3$, we define
  \begin{equation}\label{modified-drift}
  \tilde u_t(\omega, x)= \big[\big(\varphi_t(\omega, \cdot )^{-1}\big)_\ast \, u_t(\omega,\cdot) \big](x)
  \end{equation}
which is the pull-back of the field $u_t(\omega,\cdot)$ by the stochastic flow $\{\varphi_t(\omega, \cdot)\}_{t\geq 0}$. If we denote by $K_t(\omega, x)=(\nabla \varphi_t(\omega, x))^{-1}$, i.e., the inverse of the Jacobi matrix, then
  $$\tilde u_t(\omega, x) = K_t(\omega, x)\, u_t(\omega, \varphi_t(\omega, x)).$$
From this expression we see that if $u\in C([0,T], C^{l+1,\alpha})$ a.s., then one also has a.s. $\tilde u \in C([0,T], C^{l+1,\alpha})$. Moreover, if the process $u$ is adapted, then so is $\tilde u$. Now we consider the random ODE
  \begin{equation}\label{random-ode}
  \dot Y_t= \tilde u_t(Y_t),\quad Y_0= {\rm I}.
  \end{equation}
Applying the generalized It\^o formula, we see that (cf. \cite{BM, OP})
  $$X_t=\varphi_t \circ Y_t$$
is the flow of $C^{l+1,\alpha}$-diffeomorphisms associated to the SDE in \eqref{representation}.

Once we have the stochastic flow $\{X_t\}_{t\geq 0}$, we can use the second formula in \eqref{representation} to obtain a new random vector field $\hat u$. Our purpose is to show that this series of transforms have a fixed point.

From the above discussions, we see that we can fix a random element $\omega\in \Omega_0$, where $\Omega_0$ is some full measure set, and consider $\varphi_t(\omega,\cdot)$, $u_t(\omega,\cdot)$ and so on as deterministic objects. Hence in Section \ref{determ-case} we solve a deterministic fixed-point problem, and apply this result in Section \ref{local-existence} to prove the local existence of the system \eqref{representation}.

\subsection{Deterministic case}\label{determ-case}

In this section, we assume that we are given a deterministic family of diffeomorphisms $\{\varphi_t\}_{t\in [0,T]}$ of $\T^3$ satisfying $\varphi\in C\big([0,T], C^{l+2,\alpha}\big)$ and $\varphi_0= {\rm I}$. For $u\in C\big([0,T], C^{l+1,\alpha}(\T^3, \R^3)\big)$, we consider the following system:
  \begin{equation}\label{determ-system}
  \begin{cases} \aligned
  \tilde u_t(x) &= \big[(\varphi_t^{-1})_\ast u_t\big](x),\\
  \dot Y_t &= \tilde u_t(Y_t),\quad Y_0= {\rm I},\\
  X_t &= \varphi_t(Y_t),\\
  \hat u_t(x) &= \mathbf P\big[(\nabla X_t^{-1})^\ast u_0(X_t^{-1})\big](x).
  \endaligned \end{cases}
  \end{equation}
Following the arguments in \cite[Section 4]{Cons} and \cite[Section 4]{Iyer}, we shall prove that the map defined by $\hat u= \Phi(u)$ has a fixed point in
  $$\mathcal U_{\tau,U}=\bigg\{u\in C\big([0,\tau], C^{l+1,\alpha}(\T^3, \R^3)\big): \sup_{t\leq \tau} \|u_t\|_{l+1, \alpha} \leq U,\, \div(u_t)=0,\, u|_{t=0}=u_0\bigg\}$$
for some small $\tau$ and big $U$. Here is the main result of this section.

\begin{theorem}\label{thm-deterministic}
There exists $U= U(l,\|u_0\|_{l+1,\alpha})$ and $\tau=\tau(l, U,\varphi)$ such that the map $\Phi: \mathcal U_{\tau,U}\to \mathcal U_{\tau,U}$ has a unique fixed point.
\end{theorem}

We need some preparations. The following result is taken from \cite[Lemma 4.1]{Iyer}.

\begin{lemma}\label{lem-1}
If $l\geq 1$ and $\alpha\in (0,1)$, then there exists $C=C(l,\alpha)$ such that
  $$\aligned \|f\circ g\|_{l,\alpha}&\leq C\|f\|_{l,\alpha} \big(1+\|\nabla g\|_{l-1,\alpha}\big)^{l+1},\\
  %\|f\circ g_1- f\circ g_2\|_{l,\alpha}&\leq C \big(1+ \|\nabla g_1\|_{l-1,\alpha} +\|\nabla %g_2\|_{l-1,\alpha}\big)^{l+1}  \|\nabla f\|_{l,\alpha}\|g_1 -g_2\|_{l,\alpha},\\
  \|f_1\circ g_1- f_2\circ g_2\|_{l,\alpha} &\leq C \big(1+ \|\nabla g_1\|_{l-1,\alpha} +\|\nabla g_2\|_{l-1,\alpha}\big)^{l+1} \\
  &\hskip13pt \times \big[\|f_1 -f_2\|_{l,\alpha} + \big(\|\nabla f_1\|_{l,\alpha}\wedge \|\nabla f_2\|_{l,\alpha}\big) \|g_1 -g_2\|_{l,\alpha}\big].
  \endaligned $$
\end{lemma}

\begin{lemma}\label{lem-2}
Let $u\in C\big([0,T], C^{l+1,\alpha}(\T^3, \R^3)\big)$ and consider the ODE
  $$\dot Y_t= u_t(Y_t),\quad Y_0= {\rm I}.$$
Denote by $U_t= \sup_{s\leq t} \|u_s\|_{l+1,\alpha}$ and $\lambda= Y-{\rm I},\, \ell= Y^{-1} -{\rm I}$. Then there exists a continuous function $f_{l,\alpha}: [0,T]\times \R_+ \to \R_+$, which is increasing in both variables and $f_{l,\alpha}(0, \theta)=0$ for all $\theta\geq 0$, such that
  $$\|\nabla \lambda_t\|_{l,\alpha}\leq f_{l,\alpha}(t, U_t),\quad \|\nabla \ell_t\|_{l,\alpha}\leq f_{l,\alpha}(t, U_t), \quad t\leq T.$$
\end{lemma}

\begin{proof}
We first prove the case for $l=0$. Using the integral form of the ODE, it is clear that
  \begin{equation}\label{lem-2.1}
  \nabla \lambda_t= \int_0^t (\nabla u_s)(Y_s)\, (\mathbb I+ \nabla \lambda_s)\,\d s,
  \end{equation}
where $\mathbb I$ is the $3\times 3$ identity matrix. Therefore,
  $$\|\nabla \lambda_t\|_0\leq \int_0^t \|(\nabla u_s)(Y_s)\|_0 \, \|\mathbb I+ \nabla \lambda_s\|_0\,\d s \leq U_t \int_0^t (1+ \|\nabla \lambda_s\|_0)\,\d s.$$
The Gronwall inequality implies that
  \begin{equation}\label{lem-2.2}
  1+ \|\nabla \lambda_t\|_0\leq e^{tU_t}.
  \end{equation}
Now for $x,y\in \T^3,\, x\neq y$, we deduce from \eqref{lem-2.1} that
  \begin{equation}\label{lem-2.3}
  \aligned \frac{|\nabla\lambda_t(x)- \nabla\lambda_t(y)|}{|x-y|^\alpha}
  &\leq \int_0^t \frac{\big|(\nabla u_s)(Y_s(x)) - (\nabla u_s)(Y_s(y)) \big|}{|x-y|^\alpha} \, |\mathbb I+ \nabla \lambda_s(x)|\,\d s\\
  &\hskip13pt +  \int_0^t |(\nabla u_s)(Y_s(x))| \frac{\big| \nabla \lambda_s(x) - \nabla \lambda_s(y) \big|} {|x-y|^\alpha} \,\d s\\
  &=: J_1 + J_2.\endaligned
  \end{equation}
We have
  $$\aligned
  J_1 &\leq \int_0^t [\nabla u_s]_\alpha \|\nabla Y_s\|_0^\alpha  (1+\|\nabla \lambda_s\|_0)\,\d s \leq \int_0^t [\nabla u_s]_\alpha  (1+\|\nabla \lambda_s\|_0)^{1+\alpha}\,\d s\\
  &\leq \int_0^t U_s\, e^{(1+\alpha)s U_s}\,\d s \leq \frac{1}{1+\alpha}\big( e^{(1+\alpha)t U_t} -1\big),
  \endaligned$$
where the third inequality follows from \eqref{lem-2.2}. Moreover,
  $$J_2\leq  \int_0^t \|\nabla u_s\|_0 |\nabla\lambda_s|_\alpha \,\d s \leq \int_0^t U_s\, |\nabla\lambda_s|_\alpha \,\d s.$$
Substituting the estimates $J_1$ and $J_2$ into \eqref{lem-2.3}, we deduce that
  $$|\nabla\lambda_t|_\alpha\leq  \frac{1}{1+\alpha}\big( e^{(1+\alpha)t U_t} -1\big) + \int_0^t U_s\, |\nabla\lambda_s|_\alpha \,\d s.$$
Gronwall's inequality leads to
  $$|\nabla\lambda_t|_\alpha \leq \frac{e^{t U_t}}{1+\alpha}\big( e^{(1+\alpha)t U_t} -1\big).$$
Combining this estimate with \eqref{lem-2.2} and using the simple inequality $e^t-1\leq t e^t\, (t\geq 0)$, we conclude that
  $$\|\nabla\lambda_t\|_{0,\alpha} \leq 2 tU_t\, e^{(2+\alpha)t U_t},\quad t\leq T.$$

Now we prove the general case by induction. Taking the $C^{l,\alpha}$-norm in \eqref{lem-2.1} and using Lemma \ref{lem-1}, we obtain
  $$\aligned
  \|\nabla \lambda_t\|_{l,\alpha}&\leq C_l \int_0^t \|\nabla u_s\|_{l,\alpha} (1+ \|\nabla Y_s\|_{l-1,\alpha})^{l+1} (1+ \|\nabla \lambda_s\|_{l,\alpha})\,\d s\\
  &\leq C_l \int_0^t \|\nabla u_s\|_{l,\alpha} (2+ \|\nabla \lambda_s\|_{l-1,\alpha})^{l+1} (1+ \|\nabla \lambda_s\|_{l,\alpha})\,\d s.
  \endaligned$$
Using the induction hypothesis, we have
  $$\|\nabla \lambda_t\|_{l,\alpha} \leq C_l \int_0^t U_s (2+ f_{l-1,\alpha}(s, U_s))^{l+1} (1+ \|\nabla \lambda_s\|_{l,\alpha})\,\d s.$$
Again by Gronwall's inequality,
  $$1+ \|\nabla \lambda_t\|_{l,\alpha}\leq \exp\big[C_l t U_t (2+ f_{l-1,\alpha}(t, U_t))^{l+1}\big].$$
The proof of the first estimate is complete.

To prove the second assertion, note that the inverse flow $Y^{-1}_t$ can be obtained by reversing the time. More precisely, fix any $t\in (0,T]$ and consider
  $$\dot Y^t_s= -u_{t-s}(Y^t_s),\quad 0\leq s\leq t,\quad  Y^t_0= {\rm I}.$$
Then $Y^t_t= Y^{-1}_t$. Similar to the above arguments, we can prove estimates for $\lambda^t_s = Y^t_s-{\rm I}$, and hence for $\ell_t= Y^{-1}_t- {\rm I}=Y^t_t- {\rm I}=\lambda^t_t$, which only depends on the $C^{l+1,\alpha}$-norm of $u_s$ for  $s\in [0,t]$, the latter being dominated by $U_t$. In this way, we obtain the second result.
\end{proof}

We need the following key technical result, see \cite[Proposition 1]{Cons} or \cite[Corollary 3.1]{Iyer}.

\begin{lemma}\label{lem-4}
For $l\geq 1$, the operator $\mathbf W:(\ell, v)\to \mathbf P [(\mathbb I +\nabla \ell)^\ast v]$ is well defined on $C^{l,\alpha}\times C^{l,\alpha}$ with values in $C^{l,\alpha}$; moreover, there is $C>0$ depending only on $l$ and $\alpha$ such that
  $$\|\mathbf W(\ell, v)\|_{l,\alpha}\leq C (1+\|\nabla \ell\|_{l-1,\alpha}) \|v\|_{l,\alpha}.$$
\end{lemma}

Now we can prove

\begin{proposition}\label{prop-1}
There exist $U=U(l,\|u_0\|_{l+1,\alpha})>0$ and $\tau_1=\tau_1(l,U, \varphi)>0$ such that $\Phi(\mathcal U_{\tau_1, U})\subset \mathcal U_{\tau_1, U}$.
\end{proposition}

\begin{proof}
Let $U>0$ be a constant which will be determined later. We divide the proof into four steps.

\emph{Step 1}. Take $u\in \mathcal U_{T, U}$. By the definition of $\tilde u$ in \eqref{determ-system}, we have
  $$\aligned
  \|\tilde u_t\|_{l+1,\alpha}& \leq C_l \|(\nabla\varphi_t)^{-1}\|_{l+1,\alpha} \|u_t\circ \varphi_t\|_{l+1,\alpha} \\ &\leq C_l \|(\nabla\varphi_t)^{-1}\|_{l+1,\alpha} \|u_t\|_{l+1,\alpha} (1+\|\nabla\varphi_t\|_{l,\alpha})^{l+2}.
  \endaligned$$
Note that
  $$\|(\nabla\varphi_t)^{-1}\|_{l+1,\alpha} = \|(\nabla\varphi_t^{-1})\circ \varphi_t\|_{l+1,\alpha} \leq C_l \|\nabla\varphi_t^{-1}\|_{l+1,\alpha} (1+\|\nabla\varphi_t\|_{l,\alpha})^{l+2}.$$
Therefore,
  $$\tilde U_t:= \sup_{s\leq t} \|\tilde u_s\|_{l+1,\alpha}\leq C_{l,\varphi,t} \, U_t,$$
where
  \begin{equation}\label{prop-1.1}
  C_{l,\varphi,t}=\sup_{s\leq t} \big(1+\|\varphi_s\|_{l+2,\alpha}\vee \|\varphi_s^{-1}\|_{l+2,\alpha}\big)^{2l+5}.
  \end{equation}

\emph{Step 2}. Let $Y$ be the flow generated by $\tilde u$, and denote by $\ell= Y^{-1}- {\rm I}$. Then applying Lemma \ref{lem-2} with $u$ replaced by $\tilde u$ gives us
  $$\|\nabla \ell_t\|_{l,\alpha}\leq f_{l,\alpha}(t, \tilde U_t) \leq f_{l,\alpha}(t, C_{l,\varphi, t}\, U_t),\quad t\leq T.$$

\emph{Step 3}. Let $X_t=\varphi_t\circ Y_t$ and denote by $m_t = \varphi_t^{-1}- {\rm I},\,0\leq t\leq T$. Then we have
  $$X_t^{-1}= Y_t^{-1}\circ \varphi_t^{-1} = \varphi_t^{-1} + \ell_t \circ\varphi_t^{-1}= {\rm I} + m_t + \ell_t \circ\varphi_t^{-1}.$$
By Lemma \ref{lem-1} and \emph{Step 2},
  $$\aligned
  \|\nabla(\ell_t\circ \varphi^{-1}_t)\|_{l,\alpha}&\leq C_l\|(\nabla\ell_t)\circ \varphi^{-1}_t)\|_{l,\alpha} \|\nabla \varphi^{-1}_t\|_{l,\alpha}\\
  &\leq C_l\|\nabla\ell_t\|_{l,\alpha} \big(1+ \|\nabla \varphi^{-1}_t\|_{l-1,\alpha}\big)^{l+1} \|\nabla \varphi^{-1}_t\|_{l,\alpha}\\
  &\leq C_{l,\varphi, t}\, f_{l,\alpha}(t, C_{l,\varphi, t}\, U_t),
  \endaligned$$
where $C_{l,\varphi, t}$ is defined in \eqref{prop-1.1}.

Let $\tilde\ell= m+\ell\circ \varphi^{-1}$ and $C_{l,m,t}:=\sup_{s\leq t}\|\nabla m_s\|_{l,\alpha}$. Then
  \begin{equation}\label{prop-1.2}
  \|\nabla \tilde\ell_t\|_{l,\alpha}\leq \|\nabla m_t\|_{l,\alpha} + \|\nabla(\ell_t\circ \varphi^{-1}_t)\|_{l,\alpha} \leq C_{l,m,t}+ C_{l,\varphi, t}\, f_{l,\alpha}(t, C_{l,\varphi, t}\, U_t).
  \end{equation}

\emph{Step 4}. By the definition of $\hat u$ in \eqref{determ-system} and \emph{Step 3}, we have
  $$\hat u_t= \mathbf P\big[(\nabla X_t^{-1})^\ast\, u_0(X_t^{-1})\big] = \mathbf W\big(\tilde\ell_t, u_0(X_t^{-1}) \big). $$
Lemmas \ref{lem-4} and \ref{lem-1} imply that
  \begin{equation}\label{prop-1.3}
  \|\hat u_t\|_{l+1,\alpha}\leq C(1+\|\nabla \tilde\ell_t\|_{l,\alpha})\|u_0\circ ({\rm I}+\tilde \ell_t ) \|_{l+1,\alpha} \leq C_l \|u_0\|_{l+1,\alpha}(1+\|\nabla \tilde\ell_t\|_{l,\alpha})^{l+3}.
  \end{equation}
Let $U=2^{l+3}C_l\|u_0\|_{l+1,\alpha}$. Since $\varphi_0= {\rm I}$, one has $m_0\equiv 0$ which implies that $C_{l,m,t}$ decreases to 0 as $t\to 0$. Hence, by the definitions of $C_{l,m,t}$ and $f_{l,\alpha}(t, \theta)$, we see that
  $$\tau_1=\inf\big\{t>0: C_{l,m,t}+ C_{l,\varphi, t}\, f_{l,\alpha}(t, C_{l,\varphi, t}\, U) >1 \big\} >0.$$
For $u\in \mathcal U_{\tau_1, U}$, one has $U_t=\sup_{s\leq t} \|u_s\|_{l+1,\alpha} \leq U$ for all $t\leq \tau_1$. By \eqref{prop-1.2}, $\|\nabla \tilde\ell_t\|_{l,\alpha}\leq 1$ for all $t\leq \tau_1$. Thus \eqref{prop-1.3} implies that $\hat u|_{t\in [0,\tau_1]}\in \mathcal U_{\tau_1, U}$. Now it is clear that one has $\Phi(\mathcal U_{\tau_1,U})\subset \mathcal U_{\tau_1,U}$ for $U$ and $\tau_1$ defined above.
\end{proof}

The next estimate is need for establishing contraction property of $\Phi$.

\begin{lemma}\label{lem-3}
Let $u, \bar u\in C\big([0,T], C^{l+1,\alpha}(\T^3, \R^3)\big)$ be satisfying
  $$\sup_{s\leq t} \big(\|u_s\|_{l+1,\alpha}\vee \|\bar u_s\|_{l+1,\alpha}\big) \leq U_t.$$
Let $Y, \,\bar Y$ be the flows generated by $u$ and $\bar u$, respectively. Then there exists a continuous function $\bar f_{l,\alpha}:[0,T]\times \R_+\to \R_+$ which is increasing in both variables, such that
  $$\|Y_t- \bar Y_t\|_{l,\alpha}\vee \|Y_t^{-1}- \bar Y_t^{-1}\|_{l,\alpha}\leq \bar f_{l,\alpha}(t,U_t)\int_0^t \|u_s-\bar u_s\|_{l,\alpha}\,\d s.$$
\end{lemma}

\begin{proof}
We have
  $$Y_t- \bar Y_t= \int_0^t \big(u_s(Y_s) - \bar u_s(\bar Y_s)\big)\,\d s.$$
Therefore, by Lemma \ref{lem-1},
  $$\aligned
  \|Y_t- \bar Y_t\|_{l,\alpha}&\leq \int_0^t \|u_s(Y_s) - \bar u_s(\bar Y_s)\|_{l,\alpha}\,\d s\\
  &\leq C_l \int_0^t \big(1+\|\nabla Y_s\|_{l-1,\alpha} + \|\nabla \bar Y_s\|_{l-1,\alpha}\big)^{l+1}\\
  &\hskip42pt \times \big[\|u_s-\bar u_s\|_{l,\alpha} + \big(\|\nabla u_s\|_{l,\alpha} \wedge \|\nabla \bar u_s\|_{l,\alpha}\big) \|Y_s- \bar Y_s\|_{l,\alpha}\big]\,\d s,
  \endaligned$$
Lemma \ref{lem-2} implies that
  $$\|\nabla Y_s\|_{l-1,\alpha} =\|\mathbb I + \nabla\lambda_s\|_{l-1,\alpha}\leq 1+ f_{l,\alpha}(s, U_s).$$
Similarly, $\|\nabla \bar Y_s\|_{l-1,\alpha}\leq 1+ f_{l,\alpha}(s, U_s)$. Substituting these estimates into the above inequality yields
  $$\|Y_t- \bar Y_t\|_{l,\alpha} \leq 3^{l+1}C_l\int_0^t \big[1+ f_{l,\alpha}(s, U_s)\big]^{l+1} \big[\|u_s-\bar u_s\|_{l,\alpha} + U_s\|Y_s-\bar Y_s\|_{l,\alpha}\big]\,\d s.$$
From this and Gronwall's inequality we obtain the first assertion. The proof of the second one follows analogously by reversing the time.
\end{proof}

Before proving that the map $\Phi$ is a contraction in a certain space, we introduce the following property of the operator $\mathbf W$ defined in Lemma \ref{lem-4} (see \cite[Proposition 3.1]{Iyer}).

\begin{lemma}\label{lem-5}
Let $l\geq 1$ and $\ell_i, v_i\in C^{l,\alpha}$, satisfying
  $$\|\nabla\ell_i\|_{l-1,\alpha}\leq L,\quad \|v_i\|_{l,\alpha}\leq V, \quad i=1,2.$$
Then there exists $C=C(l,\alpha, L)$ such that
  $$\|\mathbf W(\ell_1,v_1)- \mathbf W(\ell_2,v_2)\|_{l,\alpha} \leq C\big( V\|\ell_1 -\ell_2\|_{l,\alpha} + \|v_1 -v_2\|_{l,\alpha} \big).$$
\end{lemma}

\begin{proposition}\label{prop-2}
Let $U$ be given in Proposition \ref{prop-1}. There exists $\tau\in (0,\tau_1]$ such that $\Phi: \mathcal U_{\tau,U}\to \mathcal U_{\tau,U}$ is a contraction with respect to the weaker norm $\|u\|_{\mathcal U_{\tau,U}}= \sup_{t\leq\tau} \|u_t\|_{l,\alpha}$.
\end{proposition}

\begin{proof}
\emph{Step 1}. Let $u_1,u_2\in \mathcal U_{\tau_1,U}$. Define $\tilde u_1, \tilde u_2$ as in \eqref{determ-system}. We have by Lemma \ref{lem-1} that
  \begin{equation}\label{prop-2.1}
  \aligned
  \|\tilde u_{1,t}- \tilde u_{2,t}\|_{l,\alpha}&\leq C_l\|(\nabla\varphi_t)^{-1}\|_{l,\alpha} \|u_{1,t}\circ\varphi_t - u_{2,t}\circ\varphi_t\|_{l,\alpha}\\
  &\leq C_l \|(\nabla\varphi_t)^{-1}\|_{l,\alpha} \|u_{1,t}-u_{2,t}\|_{l,\alpha} (1+\|\nabla \varphi_t \|_{l-1, \alpha})^{l+1}\\
  &\leq C_{l,\varphi, t} \|u_{1,t}-u_{2,t}\|_{l,\alpha},
  \endaligned
  \end{equation}
where $C_{l,\varphi, t}$ is defined in \eqref{prop-1.1}. Recall that by \emph{Step 1} in the proof of Proposition \ref{prop-1}, we have
  \begin{equation}\label{prop-2.1.5}
  \aligned
  \tilde U_t&:=\sup_{s\leq t}\big(\|\tilde u_{1,s}\|_{l+1,\alpha} \vee \|\tilde u_{2,s}\|_{l+1,\alpha}\big)\\
  &\, \leq C_{l,\varphi,t} \sup_{s\leq t}\big(\| u_{1,s}\|_{l+1,\alpha} \vee \|u_{2,s}\|_{l+1,\alpha}\big) \leq C_{l,\varphi,t} \, U.
  \endaligned
  \end{equation}

\emph{Step 2}. Let $Y_{i}$ be the flow associated to $\tilde u_{i}$ and $\ell_{i}=Y_{i}^{-1}- {\rm I},\, i=1,2$. By Lemma \ref{lem-3},
  $$  \aligned
  \|\ell_{1,t}- \ell_{2,t}\|_{l,\alpha}&= \big\|Y_{1,t}^{-1}- Y_{2,t}^{-1}\big\|_{l,\alpha}\leq \bar f_{l,\alpha}(t, \tilde U_t) \int_0^t \|\tilde u_{1,s}- \tilde u_{2,s}\|_{l,\alpha}\,\d s\\
  &\leq C_{l,\varphi, t}\bar f_{l,\alpha}(t,C_{l,\varphi, t}\, U) \int_0^t \|u_{1,s}- u_{2,s}\|_{l,\alpha}\,\d s,
  \endaligned $$
where the last inequality follows from \eqref{prop-2.1} and \eqref{prop-2.1.5}.

\emph{Step 3}. Recall that $m_t=\varphi_t^{-1} -{\rm I}$. Now for $\tilde\ell_{i,t} =m_t + \ell_{i,t}\circ \varphi_t^{-1},\, i=1,2$, one has
  \begin{equation}\label{prop-2.2}
  \aligned
  \|\tilde\ell_{1,t}- \tilde\ell_{2,t}\|_{l,\alpha}&= \big\|\ell_{1,t}\circ \varphi_t^{-1} -\ell_{2,t}\circ \varphi_t^{-1}\big\|_{l,\alpha}\\
  &\leq C_l \|\ell_{1,t}- \ell_{2,t}\|_{l,\alpha} \big(1+\|\nabla\varphi_t^{-1}\|_{l-1,\alpha} \big)^{l+1}\\
  &\leq C_{l,\varphi, t}^2 \bar f_{l,\alpha}(t,C_{l,\varphi, t}\, U) \int_0^t \|u_{1,s}- u_{2,s}\|_{l,\alpha}\,\d s.
  \endaligned
  \end{equation}
By the  definition of $\tau_1$, we have
  \begin{equation}\label{prop-2.2.5}
  \|\nabla\tilde\ell_{i,t}\|_{l,\alpha} \leq 1,\quad t\leq \tau_1,\quad  i=1,2.
  \end{equation}

\emph{Step 4}. Denote by $ X_{i,t}= \varphi_t\circ  Y_{i,t}$; then $X_{i,t}^{-1}={\rm I} + \tilde\ell_{i,t}$ and $\hat u_{i,t}= \mathbf{W}(\tilde\ell_{i,t}, u_0\circ X_{i,t}^{-1}),\, i=1,2$. By \eqref{prop-2.2.5} and Lemma \ref{lem-5}, there exists a constant $C$ depending only on $l$ and $\alpha$ such that
  \begin{equation}\label{prop-2.3}
  \|\hat u_{1,t}- \hat u_{2,t}\|_{l,\alpha} \leq C\big(V_t \|\tilde\ell_{1,t} - \tilde\ell_{2,t}\|_{l,\alpha} + \|u_0\circ X_{1,t}^{-1} - u_0\circ X_{2,t}^{-1}\|_{l,\alpha}\big),
  \end{equation}
where
  $$\aligned V_t& =\max_{i=1,2} \|u_0\circ X_{i,t}^{-1} \|_{l,\alpha} \leq C\|u_0\|_{l,\alpha} \max_{i=1,2} \big(1+ \|\nabla X_{i,t}^{-1}\|_{l-1,\alpha}\big)^{l+1}\\
  &\leq C\|u_0\|_{l,\alpha} \max_{i=1,2} \big(1+ \|\nabla\tilde \ell_{i,t}\|_{l-1,\alpha}\big)^{l+1} \leq C 2^{l+1} \|u_0\|_{l,\alpha} \leq U,
  \endaligned$$
where the third inequality follows from \eqref{prop-2.2.5}. By Lemma \ref{lem-1},
  $$\aligned \|u_0\circ X_{1,t}^{-1} - u_0\circ X_{2,t}^{-1}\|_{l,\alpha} &\leq C\big(1+ \|\nabla X_{1,t}^{-1}\|_{l-1,\alpha} + \|\nabla X_{2,t}^{-1}\|_{l-1,\alpha}\big)^{l+1} \|\nabla u_0\|_{l,\alpha} \|X_{1,t}^{-1} -X_{2,t}^{-1}\|_{l,\alpha}\\
  &\leq U \|\tilde \ell_{1,t} -\tilde \ell_{2,t}\|_{l,\alpha}.
  \endaligned$$
Therefore, substituting this estimate into \eqref{prop-2.3} yields
  $$\aligned\|\hat u_{1,t}- \hat u_{2,t}\|_{l,\alpha} &\leq CU \|\tilde \ell_{1,t} -\tilde \ell_{2,t}\|_{l,\alpha} \leq C_{l,\varphi, t}^2\, U \bar f_{l,\alpha}(t,C_{l,\varphi, t}\, U) \int_0^t \|u_{1,s}- u_{2,s}\|_{l,\alpha}\,\d s,
  \endaligned$$
where the last inequality follows from \eqref{prop-2.2}.  Consequently,
  \begin{equation*}
  \sup_{s\leq t} \|\hat u_{1,s}- \hat u_{1,s}\|_{l,\alpha} \leq C_{l,\varphi, t}^2\, U \bar f_{l,\alpha}(t,C_{l,\varphi, t}\, U)\, t \sup_{s\leq t} \|u_{1,s}- u_{1,s}\|_{l,\alpha}.
  \end{equation*}
Define
  \begin{equation}\label{prop-2.4}
  \tau:=\inf\big\{t\in (0,\tau_1]: C_{l,\varphi, t}^2\, U \bar f_{l,\alpha}(t,C_{l,\varphi, t}\, U) t > 1/2\big\}
  \end{equation}
with the convention that $\inf\emptyset= \tau_1$. Then it is clear that the map $\Phi$ is a contraction on $\mathcal U_{\tau,U}$ with respect to the norm $\|u\|_{\mathcal U_{\tau,U}}= \sup_{t\leq\tau} \|u_t\|_{l,\alpha}$.
\end{proof}

Finally we are ready to prove Theorem \ref{thm-deterministic}.

\begin{proof}[Proof of Theorem \ref{thm-deterministic}]
With the above preparations, the proof is the same as that in \cite{Iyer}. The existence of a fixed point of $\Phi$ follows by successive iteration. We define $u_{n+1}=\Phi(u_n)$. The sequence converges strongly with respect to the $C^{l,\alpha}$-norm. Since $\mathcal U_{\tau,U}$ is closed and convex, and the sequence $\{u_n\}_{n\geq 1}$ is uniformly bounded in the $C^{l+1,\alpha}$-norm, it must have a weak limit $u\in \mathcal U_{\tau,U}$. Since $\Phi$ is continuous with respect to the weaker $C^{l,\alpha}$-norm, this limit must be a fixed point of $\Phi$, and thus a solution of the deterministic system \eqref{determ-system}.
\end{proof}

\subsection{Local existence of \eqref{representation}}\label{local-existence}

Let $\Omega_0\subset \Omega$ be a full measure set such that for all $\omega\in \Omega_0$, $\varphi_t(\omega,\cdot)$ is a $C^{l+2,\alpha}$-diffeomorphism on $\T^3$ for any $t>0$. Let $\Phi_\omega$ be the map in Theorem \ref{thm-deterministic} associated to $\{\varphi_t(\omega, \cdot)\}_{t\geq 0}$. By Theorem \ref{thm-deterministic}, there exists $\tau(\omega)>0$ and a time-dependent vector field $u_\cdot (\omega)\in \mathcal U_{\tau(\omega), U}$ such that $\Phi_\omega(u_\cdot (\omega))= u_\cdot (\omega)$. It follows from the definition of $\tau$ in \eqref{prop-2.4} that it is a stopping time. Then the random vector field $u(\omega): [0,\tau(\omega))\times \T^3 \ni (t,x)\to u_t(\omega, x)\in \R^3$ is a solution to  \eqref{representation}.

\section{Discussions}

The inverse flow $A_{t}(x) =X_{t}^{-1}(x) $ is a
random vector field, solution of the stochastic transport equation%
\[
\d A_{t}( x) +u_{t}( x) \cdot \nabla A_{t}(x)\, \d t+\sum_{k}\sigma _{k}( x) \cdot \nabla A_{t}(x) \circ \d W_{t}^{k}=0.
\]%
This equation does not contain stretching terms of the form $A_{t}(
x) \cdot \nabla u_{t}( x)\, \d t$ and $A_{t}( x) \cdot \nabla \sigma _{k}( x) \circ \d W_{t}^{k}$, hence the
quantity $A_{t}( x) $ is only transported. Therefore we do not
expect a blow-up of $A_{t}( x) $ itself. We may however expect,
in analogy with shocks appearing in nonlinear transport equations (like
Burgers equation), that space derivatives of $A_{t}( x) $ may
blow-up. This is the potential mechanism which could lead to blow-up in the
formula%
\[
u_{t}( x) =\mathbf{P}\left[ \left( \nabla A_{t}\right)
^\ast u_{0}\circ A_{t}\right] \left( x\right)
\]%
as discussed in \cite[Section 5]{Cons}.

The question posed by the presence of noise is:\ could the noise prevent or
mitigate blow-up of $\nabla A_{t}$? If $u_{t}( x) $, in the SPDE
above, would be given (passive field $A_{t}$) and deterministic, several
results of regularization due to noise have been proved for similar
equations, for instance the absence of shocks for the scalar transport
equation with $u$ of class $L^{q}\left( 0,T;L^{p}\left( \mathbb{R}^{d},%
\mathbb{R}^{d}\right) \right) $ with $\frac{d}{p}+\frac{2}{q}<1$ (see
\cite{FF}), or the absence of singularity for a passive magnetic
field proved in \cite{FMN, FO} under various assumptions. The intuitive reason is that noise
prevents $A_{t}( x) $ to stretch for too much time around the
more singular points of $u_{t}( x) $, because $A_{t}(x) $ is continuosly randomly displaced. However, no result of this
form has been proved until now in the case when $u_{t}( x) $ is
random (see \cite{DR} for a related work), as it is in the nonlinear case; the obstruction is not technical but
conceptual: the singularities of $u_{t}( x) $ move accordingly to
noise and to $A_{t}( x) $ itself, hence there is no
straightforward reason why noise should displace $A_{t}( x) $ to
avoid those singularities.

Thus the question of singularities remains open, as it is in the
deterministic case but here, thanks to the noise, new intuitions may develop.

\bigskip

\noindent \textbf{Acknowledgements.} The second author is supported by the National Natural Science Foundation of China (Nos. 11431014, 11571347), the Seven Main Directions (Y129161ZZ1) and the Special Talent Program of the Academy of Mathematics and Systems Science, Chinese Academy of Sciences.

\end{document}